\documentclass[12pt,reqno]{amsart}
\usepackage{tikz-cd}
\usepackage{dsfont,amsmath, amssymb, graphicx, hyperref}
\usepackage{verbatim} 
\usepackage{amssymb, amsmath, amsthm, amsfonts} 
\usepackage[mathscr]{euscript}
\usepackage{enumerate}
\renewcommand{\pmod}[1]{\,(\textup{mod}\,#1)}

\newcommand{\LL}{\mathcal{L}}

\newcommand{\ZZ}{\mathbb{Z}}

\newcommand{\RR}{\mathbb{R}}

\newcommand{\CC}{\mathbb{C}}
\newcommand{\NN}{\mathbb{N}}

\newtheoremstyle{dotless}{}{}{\itshape}{}{\bfseries}{}{ }{}
\theoremstyle{dotless}

\newtheorem{theorem}{Theorem}[section]
\newtheorem{lemma}[theorem]{Lemma}
\newtheorem{corollary}[theorem]{Corollary}

\title{A binary quadratic Titchmarsh divisor problem}
\author{Junxian Li}
\address{Department of Mathematics, University of Illinois, 1409 West Green Street, Urbana, IL 61801, USA.}
\email{jli135@illinois.edu}
\begin{document}
\maketitle
\begin{abstract}
	We consider a binary quadratic variant of the Titchmarsh divisor problem and give an asymptotic formula for $\sum_{p^2+q^2\leq N} \tau(p^2+q^2+1)$, where $p,q$ are primes.
\end{abstract}
\section{Introduction}
Let $\tau(n)=\sum_{d\mid n} 1$ be the divisor function. The Titchmarsh divisor problem is concerned with finding an asymptotic formula for the average 
\begin{align}
\sum_{p\leq x} \tau(p-1),
\end{align}
where $p$ belongs to the set of primes. Under the Generalized Riemann Hypothesis (GRH), Titchmarsh \cite{titchmarsh1930divisor} proved that 
\begin{align}\label{tichmarsh}
\sum_{p\leq x} \tau(p-1)= \frac{\zeta(2)\zeta(3)}{\zeta(6)} x+O\left( \frac{x\log\log x}{\log x}\right).
\end{align}
Linnik \cite{Linnik} proved \eqref{tichmarsh} unconditionally using his dispersion method. Later, Halberstam \cite{Halberstam} gave a short proof using the Bombieri-Vinogradov theorem on primes in arithmetic progressions. Bombieri, Friedlander and Iwaniec \cite{BFI} as well as Fouvry \cite{Fouvry} improved \eqref{tichmarsh} to 
\begin{align}\label{bfi}
\sum_{p\leq x}\tau(p-1)= \frac{\zeta(2)\zeta(3)}{\zeta(6)}x+c\operatorname{Li}(x)+O\left( \frac{x}{(\log x)^A}\right),
\end{align}
for some constant $c$ and any $A$, where $\operatorname{Li}(x)=\int_2^x \frac{1}{\log t}dt$. Most recently, 
Drappeau \cite{Drappeau} gave a power saving in the error in \eqref{bfi} under GRH. For primes in arithmetic progressions, Felix \cite{Felix} established a formula for 
\begin{align}
\sum_{ \substack{p\leq x\\ p \equiv a\pmod k}} \tau\left( \frac{p-a}{k}\right) =c_{k,a} x+O_k\left(\frac{x}{\log x}\right),
\end{align} for some constant $c_{k,a}$. A quadratic analogue of the Titchmarsh problem was considered by Xi \cite{Xi}, where he obtained the correct order of magnitude given by
\begin{align}
x\ll \sum_{p\leq x} \tau(p^2+1)  \ll x.
\end{align}
 
In this paper, we obtain an asymptotic formula for  
\begin{align*}
\sum_{p^2+q^2\leq N} \tau(p^2+q^2+1).
\end{align*}

\begin{theorem}
	For $N$ large enough, we have
	\begin{align}
	\sum_{p^2+q^2\leq N} \tau(p^2+q^2+1)= \frac{\pi}{4}\prod_{p>2} \left( 1- \frac{1+3p\left(\frac{-1}{p}\right)}{(p-1)^2p}\right) \frac{N}{\log N} \left(1+O\left(\frac{(\log\log N)^2}{\log N}\right)\right),
	\end{align}where $p, q$ belong to the set of primes.
\end{theorem}
A related question is the Hardy-Littlewood problem concerning asymptotic formulas for
\begin{align}
\sum_{p\leq N} r(N-p) \text{ or } \sum_{p\leq N} r(p-a),
\end{align} where $r(n)$ is the number of ways of writing $n$ as the sum of two squares. This was solved in the works of Hooley \cite{Hooley} under GRH.  Unconditional proofs were given by Linnik \cite{Linnik2} and Bredihin \cite{Bredihin} using the ``dispersion method''. More recently, Friedlander and Iwaniec gave a shorter proof in \cite{Friedlander}. Greaves \cite{Greaves} considered the number of solutions to $N=p^2+q^2+x^2+y^2$ and gave the lower bound with the right order of magnitude. Later Plaksin \cite{Plaksin} obtained an asymptotic formula of the number of solutions to $N=p^2+q^2+x^2+y^2$.

Let us fix some notation: We use the relation $a\sim A$ to denote $A\leq a\leq 2A$. The arithmetic function $\omega(n)$ denotes the number of distinct prime divisors of $n$. For a prime $p$ and natural numbers $\alpha$ and $n$, we write $p^\alpha\| n$ if $p^\alpha\mid n$ but $p^{\alpha+1}\nmid n$. The letters $p$ and $q$ denote primes, the expression $e(x)$ denotes $\exp(2\pi i x)$, and $(a,b,c)$ denotes $ \gcd (a,b,c)$. Finally, for an odd integer $d$, let $$d^*= \left(\frac{-1}{d} \right) d=\left\{\begin{array}{ll}
	d,& d\equiv 1\pmod 4,\\
	-d, & d\equiv 3\pmod 4.
\end{array}\right.$$

\section{Outline of the proof}
\begin{lemma}\label{divisorFormula}
	\begin{align}
	\tau(n)= 2\sum_{\substack{d\mid n\\d\leq \sqrt{n}}} 1-\mathds{1}(n=\square),
	\end{align}
	where $\mathds{1}(n=\square)$ vanishes unless $n$ is a square, in which case it is $1$.
\end{lemma}
\begin{lemma}\label{RepSumSqure}
Let $r(n)$ be the number of representations of $n$ as a sum of two squares. Then $$r(n)= 4\sum_{d\mid n} \chi(d),$$ where $\chi$ is the non-principal character modulo $4$, and thus $$r(n)\ll\tau(n)\ll n^{\epsilon}.$$ 
\end{lemma}

Let $Z=\sqrt{N+1}(\log N)^{-A}$, for some sufficently large constant $A$ to be chosen later. From Lemma \ref{divisorFormula} and \ref{RepSumSqure}, we have 
\begin{align*}
\sum_{p^2+q^2\leq N} \tau(p^2+q^2+1)&=2 \sum_{p^2+q^2\leq N} \sum_{ \substack{p^2+q^2+1\equiv 0\pmod d\\ d\leq \sqrt{p^2+q^2+1}}} (1- s(p^2+q^2+1))\\
&= 2\sum_{p^2+q^2\leq N} \sum_{\substack{p^2+q^2\equiv -1 \pmod d\\ d\leq \sqrt{p^2+q^2+1}}} 1+O(\sum_{p^2+q^2\leq N}\sum_{p^2+q^2+1=\square}1)\\
&= 2\sum_{d\leq \sqrt{N+1}} \sum_{ \substack{d^2-1\leq p^2+q^2 \leq N\\ p^2+q^2 \equiv -1 \pmod d}}1+ O(\sum_{n\leq \sqrt{N}} r(n^2-1))\\
&=2\sum_{d\leq \sqrt{N+1}} \sum_{ \substack{ d^2-1\leq p^2+q^2 \leq N\\ p^2+q^2 \equiv -1 \pmod d}}1+O(N^{1/2+\epsilon})\\
&= 2\sum_{d\leq Z} \sum_{\substack{d^2-1\leq p^2+q^2\leq N\\ p^2+q^2\equiv -1 \pmod d}} 1+2\sum_{Z< d\leq \sqrt{N+1}} \sum_{\substack{d^2-1\leq p^2+q^2\leq N\\ p^2+q^2\equiv -1 \pmod d}} 1+O(N^{1/2+\epsilon})\\
&:=M_1+M_2+ O\left(N^{1/2+\epsilon}\right),
\end{align*}
where 
\begin{align}
M_1&=2\sum_{d\leq Z}\sum_{\substack{d^2-1\leq p^2+q^2\leq N\\ p^2+q^2\equiv -1 \pmod d}} 1,\label{M1}\\
M_2&=2\sum_{Z< d\leq \sqrt{N+1}} \sum_{\substack{d^2-1\leq p^2+q^2\leq N\\ p^2+q^2\equiv -1 \pmod d}} 1.\label{M2}
\end{align}
We show that $M_1$ gives the main term in Section \ref{Prelim} and Section \ref{EvalM1}, and that $M_2$ contributes to the error term in Section \ref{EvalM2} and Section \ref{ExpSumEstProof}. Estimates for $M_1$ are similar to the main term estimate of Plaksin \cite{Plaksin}. Assuming some preliminary results in Section \ref{Prelim}, we obtain an asymptotic formula for $M_1$ in Section \ref{EvalM1}. Now we are left to prove an upper bound for $M_2$. Plaksin used Hooley's method, as well as Linnik's dispersion method to study distribution of $u^2+v^2\leq N$ in arithmetic progressions with difference $d$ for $d\leq N^{3/4-\epsilon}$. Instead, we use upper bound sieve weights and separate $p$ and $q$ by introducing a smooth function. After applying the Possion summation formula, we are left with the problem of bounding an exponential sum of the form 
\begin{align*}
E(e_1,e_2,h_1,h_2,d)=\sum_{ \substack{e_1^2 u^2 +e_2^2v^2 \equiv -1 \pmod d\\ (uv,d)=1}} e\left( \frac{uh_1+vh_2}{d}\right).
\end{align*} We assume an upper bound for $E(e_1,e_2,h_1,h_2,d)$ in Section \ref{EvalM2} and prove the bound in Section \ref{ExpSumEstProof}.

\section{Preliminaries}\label{Prelim}
Let $\pi(x)=\#\{p\leq x\}$ and $\pi(x,d,u)=\#\{p \leq x: p \equiv u\pmod d\}$. 
\begin{lemma}[Barban-Davenport-Halberstam]\label{BVsquare}
	For any fixed $C>0$, any $x(\log x)^{-C} \leq Q\leq x$, we have 
	\begin{align*}
	\sum_{d\leq Q} \sum_{\substack{u=1\\(u,d)=1}}^d \left( \pi(x,d,u)- \frac{\pi(x)}{\phi(d)}\right)^2 \ll_C xQ \log x
	\end{align*}
\end{lemma}
\begin{proof}
	This can be found in Chap 29 of Davenport \cite{Davenport}.
\end{proof}
\begin{lemma}
	Let  $d$ be a fixed odd integer. For any fixed $u$, the number of solutions $v$ to the equation $$u^2+v^2+1\equiv 0\pmod d$$ is bounded by $\tau(d)$.
\end{lemma}
\begin{proof}
	For $d=p$, there are either $0$ or $2$ solutions for $v$ depending $u^2+1$ on whether is a square or not. Suppose $v$ is a solution to $v^2+u^2+1\equiv0\pmod {p^k}$. Then the solution to $v'^2+u^2+1=0\pmod {p^{k+1}}$ is given by $v'=p^kt+v$, where $t$ is determined by $2tu+ \frac{u^2+v^2+1}{p^k}\equiv0\mod p$. Thus for $d=p^k$ there are at most $2$ solutions to the equation $u^2+v^2+1\equiv 0 \pmod {p^k}$. The lemma follows by multiplicativity. 
\end{proof}
\begin{lemma}\label{p2q2leqN}
	\begin{align*}
	\sum_{p^2+q^2\leq N} 1= \pi N (\log N)^{-2}\left(1+ O\left( \log\log N (\log N)^{-1}\right)\right).
	\end{align*}
\end{lemma}
\begin{proof}
	This is Lemma 11 in \cite{Plaksin}. We reproduce it here for convenience. 
	The terms with $p\leq Z=\sqrt{N}(\log N)^{-A}$ can be bounded by 
	\begin{align*}
	\sum_{p\leq Z} \sum_{q\leq \sqrt{N-p^2}} 1\ll \frac{Z}{\log Z} \frac{\sqrt{N}}{\log N} \ll N (\log N)^{-A}.
	\end{align*}
	If $p\geq Z$, then $\log p\gg \log Z= \log \sqrt{N} + O\left(\log\log N\right)$. Since $p\leq \sqrt{N}$, we have $\log p = \frac{1}{2}\log N (1+ O\left( \frac{\log\log N}{\log N}\right))$, it follows that
	\begin{align*}
	\sum_{p^2+q^2 \leq N} 1 &= \sum_{ Z\leq p\leq \sqrt{N}} \sum_{ Z\leq q\leq \sqrt{N-p^2}}1+O ( N (\log N)^{-A})\\
	&= 2 \left( \frac{1}{2} \log N \right)^{-2}\sum_{ Z \leq p \leq \sqrt{N/2}} \log p \log q \left(1+ O\left( \frac{\log\log N}{\log N}\right)\right)+ O \left( N ( \log N)^{-A}\right).
	\end{align*}
	The conclusion follows from the following calculation
	\begin{align*}
	&\sum_{Z\leq p \leq \sqrt{N/2}} \log p \sum_{Z\leq q\leq \sqrt{N-p^2}} \log q\\&= \sum_{Z\leq p \leq \sqrt{N/2}} \log p(\sqrt{N-p^2}-Z)(1+ O(\sqrt{N} e^{-\sqrt{\log N}})) \\
	& = \sum_{Z\leq p \leq \sqrt{N/2}} \log p \sqrt{N-p^2} + O \left( N e^{-\sqrt{\log N}}\right)+ O \left( Z \sqrt{N}\right)\\
	&= \sum_{2\leq p \leq \sqrt{N/2}} \log p \sqrt{N-p^2} + O\left( N (\log N)^{-A'}\right)\\
	& = \int_{0}^ { \sqrt{N/2}} \sqrt{N-x^2}dx (1+ O(e^{-\sqrt{\log Z}}))+O\left( N (\log N)^{-A}\right)\\
	&= \frac{\pi}{8}N + O \left( N (\log N)^{-A}\right).
	\end{align*}
\end{proof}
\begin{lemma}
	Let $\ell $ be an odd prime. Then for $(a,p)=1$, 
	\begin{align*}
	\sum_{u=0}^{p-1} e\left( \frac{a u^2}{\ell}\right)= \left( \frac{a}{\ell }\right) \sqrt{\left( \frac{-1}{\ell}\right)\ell }= \left( \frac{a}{\ell}\right) \sqrt{\ell^*}.
	\end{align*}
	
\end{lemma}
\begin{proof}
	This can be found in Proposition 6.3.1 and Theorem 1 in \cite[Chap 5]{RosenBook}.
\end{proof}
Let $s(d)$ denote the number of solutions $(u,v)$ to 
\begin{align}\label{sndef}
u^2 +v^2 \equiv-1 \pmod d, (uv,d)=1, 1\leq u, v\leq d. 
\end{align}
\begin{lemma}\label{rprime}
	Let $\ell$ be an odd prime. Then we have
	$$s(\ell)= \ell-2-3\left( \frac{-1}{\ell}\right), s(\ell^{k+1})=\ell^{k}s(\ell).$$
	and from the multiplicativity of $s(d)$, we have
	$$s(d)\leq d \prod_{p\mid d} \left(1+ \frac{1}{p}\right).$$
\end{lemma}
\begin{proof}
	By orthogonality of the characters, we have 
	\begin{align*}
	s(\ell)&=\frac{1}{\ell}\sum_{a=0}^{\ell-1}\sum_{u=1}^{\ell-1}\sum_{v=1}^{\ell-1} e\left( \frac{a(u^2+v^2+1)}{\ell}\right)\\
	&=\frac{(\ell-1)^2}{\ell}+\frac{1}{\ell}\sum_{a=1}^{\ell-1}\left(\sum_{u=1}^{\ell-1}e \left( \frac{au^2}{\ell}\right)\right)^2e\left( \frac{a}{\ell }\right)\\
	&= \frac{(\ell-1)^2}{\ell}+\frac{1}{\ell}\sum_{a=1}^{\ell-1}\left( \left ( \frac{a}{\ell}\right) \sqrt{\ell^*} -1\right)^2 e\left(\frac{a}{\ell }\right)\\
	&= \frac{(\ell-1)^2}{\ell}+\frac{1}{\ell}\sum_{a=1}^{\ell-1} \left(  \ell^* -2 \left( \frac{a}{\ell}\right) \sqrt{\ell^* }+1\right)e\left( \frac{a}{\ell }\right)\\
	&=\frac{(\ell-1)^2}{\ell}-\left( \frac{-1}{\ell}\right)  -\frac{1}{\ell}-2\frac{1}{\ell} \sqrt{\ell^*}\sum_{a=1}^{\ell-1}\left( \frac{a}{\ell }\right)e \left( \frac{a}{\ell }\right)\\
	&=\ell -2-3 \left( \frac{-1}{\ell}\right).
	\end{align*}
	If $(u,v)$ is a solution to $u^2+v^2 +1=0\pmod {\ell^{k}}$, then $u'=u+t\ell^{k}$, $1\leq t\leq p$ determines $v'= v+ m\ell ^k$ as $  2m v\equiv \frac{-1 -u'^2-v^2}{\ell^k} \pmod \ell$. Thus $s(\ell^{k+1})=\ell^{k}s(\ell)$ and $s(d)\leq d\prod_{p\mid d}(1+\frac{1}{p})$.
\end{proof}
\begin{lemma}\label{AssympRd}
	\begin{align*}
	\sum_{d\leq Z} \frac{s(d)}{\phi(d)^2}=  \frac{1}{4}\prod_{p>2} \left(1- \frac{1+3 p\left( \frac{-1}{p}\right)}{(p-1)^2p}\right)\log N \left(1+O\left( \frac{(\log\log N)^2}{\log N}\right)\right).
	\end{align*}
\end{lemma}

\begin{proof}
	First note that $s(d)$ is multiplicative and the terms with $p=2$ or $q=2$ can be bounded by $O ( \sqrt{N})$. Thus we can assume $2\nmid d$.  
	From Perron's formula, we have
	\begin{align*}
	\sum_{d\leq x} \frac{s(d)}{\phi(d)^2}= \frac{1}{2\pi i}\int_{\kappa-iT}^{\kappa+iT} f(s) \frac{x^s}{s}ds+R(T),
	\end{align*}
	where \begin{align*}
	f(s)& = \sum_{d=1}^\infty\frac{s(d)}{\phi(d)^2d^s},\\
	R(T)&\leq \frac{x^\kappa}{T}\sum_{n=1}^\infty \frac{s(n)}{\phi(n)^2 n^\kappa|\log x/n|}.
	\end{align*}
	By applying Lemma \ref{rprime}, we obtain
	\begin{align*}
	f(s)=\prod_{p>2} \left(1+ \sum_{k=1}^\infty\frac{s(p^k)}{\phi(p^k)^2p^{ks}}\right)=&\prod_{p}\left( 1+ \sum_{k=1}^\infty \frac{s(p^k)}{\phi(p^k)^2p^{ks}}\right)\\
	=& \prod_{p>2} \left(1+ \sum_{k=1}^\infty \frac{p-1-1-3\left(\tfrac{-1}{p}\right)}{p^{k-1}(p-1)^2p^{ks}} \right)\\
	=& \prod_{p>2} \left(1+ \frac{p-1-1-3\left(\tfrac{-1}{p}\right)}{(p-1)^2} \frac{p^{-s}}{1-p^{-s-1}}\right)\\
	=& \prod_{p>2} \left( 1-p^{-s-1}\right)^{-1} \left( 1- \frac{1+3p\left( \tfrac{-1}{p}\right)}{(p-1)^2p^{s+1}}\right)\\
	=&: \zeta(1+s)(1-2^{-s-1}) G(s).
	\end{align*}
	It can be seen that $G(s)$ is entire for $\Re(s)>-1$ and $f(s)$ converges absolutely when $\Re(s)>0$. Let $\kappa= c_1/\log x$. 
	Moving the line of integration from $\Re(s)=\kappa$ to $\Re(s)=-c/\log T$, passing the pole of $\zeta(s+1)$ at $s=0$, we see that 
	\begin{align*}
	\sum_{d\leq x} \frac{s(d)}{\phi(d)^2}=\frac{1}{2} \prod_{p>2} \left(1- \frac{1+3 p\left( \frac{-1}{p}\right)}{(p-1)^2p}\right) \log x+R(T)+H(T),
	\end{align*}
	where \begin{align}
	R(T) &\leq \frac{x^2}{T}\sum_{n=1}^\infty \frac{s(n)}{\phi(n)^2n^2|\log x/n|},\\ 
	H(T)&\leq \int_{-c/\log T-iT}^{\kappa-iT} f(s) \frac{x^s}{s}ds+ \int_{-c/\log T+iT}^{\kappa+iT} f(s) \frac{x^s}{s}ds.
	\end{align}
	Since $s(n)\leq n \prod_{p\mid n} (1+ \frac{1}{p})$, we have that 
	\begin{align*}
	R(T)&\ll \frac{x^\kappa}{T} + \frac{x^\kappa}{T} \sum_{\frac{x}{2}\leq n \leq 2x} \frac{s(n)}{\phi(n)^2 n^\kappa }\frac{x}{|n-x|}\\
	&\ll  \frac{x^\kappa}{T}+ \frac{(\log\log x)^2}{T}\log x.
	\end{align*}
	Since $f(s)\ll \log |\Im s|$ when $\Re(s)\geq -c/\log T$, we see that
	\begin{align*}
	H(T) \ll (\log T)^2 \frac{x^{\kappa}}{T}.
	\end{align*} 
	We also have 
	\begin{align*}
	\int_{-c/\log T-iT}^{-c/\log T+iT} f(s) \frac{x^s}{s}ds \ll x^{-c/\log T} (\log T)^2.
	\end{align*}
	Taking $T=(\log x)^5$ gives 
	\begin{align*}
	\sum_{d\leq Z} \frac{s(d)}{\phi(d)^2} = \frac{1}{4}\prod_{p>2} \left(1- \frac{1+3 p\left( \frac{-1}{p}\right)}{(p-1)^2p}\right)\log N \left(1+ \frac{(\log\log N)^2}{\log N}\right).
	\end{align*}
\end{proof}
\section{Evaluation of $M_1$ }\label{EvalM1}

We first extract the main term in $M_1$. Note that
the terms with $p$ or $q\leq Z$ can be bounded by 
\begin{align*}
\sum_{\substack{p\leq Z, q\\ {p^2+q^2 \leq N} }} \sum_{\substack{d<Z\\ d\mid p^2+q^2+1}}1
& \ll \left(\sum_{p\leq Z,q} 1 \right)^{1/2} \left( \sum_{p\leq Z,q \leq \sqrt{N}}\left(\sum_{\substack{d<Z\\ d\mid p^2+q^2+1}} 1\right)^2\right)^{1/2}\\
& \ll \pi( Z)^{1/2} \pi(\sqrt{N})^{1/2} \left( \sum_{n\leq N+1}\tau^2(n)\sum_{\substack{p^2+q^2+1=n\\p\leq Z,q\leq \sqrt{N}}}1\right)^{1/2}\\
& \ll \pi( Z)^{1/2} \pi(\sqrt{N})^{1/2} \left( \sum_{n\leq N+1}\tau^2(n)r(n-1)\right)^{1/2}\\
& \ll \pi( Z)^{1/2} \pi(\sqrt{N})^{1/2} \left( \sum_{n\leq N+1}\tau^2(n)\tau(n-1)\right)^{1/2}\\
& \ll \pi(Z)^{1/2}\pi (\sqrt{N})^{1/2}\left( \sum_{n\leq N+1} \tau^4(n) \sum_{n\leq N} \tau^2(n)\right)^{1/4}\\
& \ll \left(Z \sqrt{N} N \log^{10} N \right)^{1/2}\\
& \ll N (\log N)^{-A/2+5}.
\end{align*}
Thus with $A'=-A/2+5$, 
from \eqref{M1}, we have 
\begin{align}
M_1&=2\sum_{d\leq Z} \sum_{\substack{u^2+v^2\equiv -1\pmod d\\ u,v\leq d}} \sum_{ \substack{p\equiv u\pmod d\\ q\equiv v\pmod d\\ d^2-1\leq p^2+q^2 \leq N}}1\nonumber\\
&=2\sum_{d\leq Z} \sum_{\substack{u^2+v^2\equiv -1\pmod d\\ u,v\leq d}} \sum_{ \substack{p\equiv u\pmod d\\ q\equiv v\pmod d\\  p^2+q^2 \leq N\\ Z<p, Z<q}}1+ O\left( N (\log N)^{-A'}\right). 
\end{align}
When $d\leq Z<p$, we must have $(p,d)=1$. Thus,
\begin{align*}
M_1
&=2\sum_{d\leq Z} \sum_{\substack{u^2+v^2\equiv -1\pmod d\\ u,v\leq d}} \sum_{\substack{p\equiv u\pmod d\\ q\equiv v\pmod d\\ p^2+q^2\leq N\\ Z< p\\ Z< q}}1+ O\left( N (\log N)^{-A'}\right)\\
&=2\sum_{d\leq Z} \sum_{\substack{u^2+v^2 \equiv -1\pmod d\\ (uv,d)=1\\ u,v\leq d}} \sum_{\substack{p\equiv u\pmod d\\ q\equiv v\pmod d\\ p^2+q^2\leq N}}1+ O\left( N (\log N)^{-A'}\right).
\end{align*}

Let $\Omega= \sqrt{N}(\log N)^{-5}$. Then, we can cover the region $G:=\{(p,q): p^2+q^2\leq N\}$ with $\ll (\log N)^{10}$ squares of the form $X_i\leq p \leq X_i+\Omega$ and $Y_j\leq q\leq Y_j+\Omega$, $i,j \ll (\log N)^5$, and the boundary of $G$ denoted by $\partial G$ can be covered with $\ll (\log N)^5$ squares. The contribution from $(p,q)\in \partial G$ can be bounded by
\begin{align}\label{boundary}
\sum_{d\leq Z} \sum_{\substack{u^2+v^2 \equiv -1 \pmod d\\ (uv,d)=1\\ u,v\leq d}} \sum_{\substack{(p,q)\in \partial G\\ p\equiv u\pmod d\\ q\equiv v\pmod d}} 1
&\ll \sum_{d\leq Z} \sum_{\substack{u^2+v^2\equiv -1\pmod d\\ (uv,d)=1}} (\log N)^5  \left( \frac{\Omega}{d}\right)^2\nonumber\\
& \ll N (\log N)^{-5}\sum_{d\leq Z} \sum_{ \substack{u^2+v^2 \equiv -1\pmod d\\ (uv,d)=1\\ u,v\leq d}} \frac{1}{d^2}\nonumber\\
& \ll N(\log N)^{-5} \sum_{2^k\leq Z} \frac{2^k}{2^{2k}}\sum_{\substack{d\leq Z\\ (d,2)=1}} \frac{\tau(d)\phi(d)}{d^2}\nonumber\\
& \ll N (\log N)^{-5} \sum_{d\leq Z} \frac{\tau(d)}{d}\nonumber\\
& \ll N(\log N)^{-5}(\log N)^2\nonumber\\
& \ll N (\log N)^{-3}.
\end{align}
Let $\Delta_x(\Omega,d,u)=\pi(x+\Omega,d,u)-\pi(x,d,u)$, and $E_x(\Omega,d,u):= \Delta_x(\Omega,d,u)- \frac{\Delta_x(\Omega)}{\phi(d)}$, where $\Delta_x(\Omega)= \pi(x+\Omega)-\pi (x)$.
For $(p,q)$ inside $G$, we have 
\begin{align*}
&\sum_{d\leq Z} \sum_{ \substack{u^2+v^2 \equiv -1\pmod d\\ (uv,d)=1\\ u,v\leq d}} \sum_{X_i}\sum_{Y_j}\sum_{\substack{X_i\leq p \leq X_i+\Omega\\ p\equiv u \pmod d}} 1\sum_{\substack{Y_j \leq p \leq Y_j+\Omega\\  q\equiv v\pmod d}}1\\
&=\sum_{d\leq Z} \sum_{ \substack{u^2+v^2 \equiv -1 \pmod d\\ (uv,d)=1\\ u,v \leq d}} \sum_{X_i,Y_j} \left(\frac{\Delta_{X_i}(\Omega)}{\phi(d)}+E_{X_i}(\Omega,d,u)\right)\left(\frac{\Delta_{Y_j}(\Omega)}{\phi(d)}+E_{Y_j}(\Omega,d,v)\right)\\
& =\sum_{d\leq Z} \frac{1}{\phi(d)^2}\sum_{ \substack{u^2+v^2 \equiv -1\pmod d\\ (uv,d)=1\\ u,v\leq d}} \sum_{X_i,Y_j} \Delta_{X_i}(\Omega,d,u)\Delta_{Y_j}(\Omega,d,v)+E',
\end{align*}
where 
\begin{align*}
E'\ll \sum_{d\leq Z} \frac{\Omega}{d} \sum_{\substack{u^2+v^2\equiv =1\pmod d \\(uv,d)=1\\ u,v\leq d}} \sum_{X_i,Y_j}|E_{X_i} (\Omega,d,u)| + | E_{Y_j}(\Omega,d,v)|,
\end{align*}
where we have used the fact that $\frac{\Delta_{X_i}(\Omega)}{\phi(d)},E_{X_i}(\Omega,d,u),\frac{\Delta_{Y_i}(\Omega)}{\phi(d)},E_{Y_i}(\Omega,d,u)\ll \frac{\Omega}{d}$ since $d\leq Z\leq \Omega$. 
For a fixed $u$, we have that for odd $ d$, 
\begin{align*}
\sum_{\substack{v^2\equiv -1-u^2 \pmod d\\ v\leq d}}1\ll \prod_{p\mid d} 2\ll 2^{\omega(d)} \ll \tau(d).
\end{align*}
Consequently, 
\begin{align}\label{errorBV}
E'&\ll \Omega\sum_{X_i,Y_j} \sum_{k\leq \log Z}\sum_{d\leq Z} \left(\frac{\tau(d)}{d} \sum_{(u,d)=1} |E_{X_i}(\Omega,d,u)|+ \sum_{(v,d)=1} |E_{Y_j}(\Omega,d,v)|\right)\nonumber\\
&\ll \Omega (\log N)^{11} \max_{X\in\{X_i,Y_j\}}\left(\sum_{d\leq Z} \frac{(\tau(d))^2}{d^2} \sum_{d\leq Z} \left(\sum_{ (u,d)=1} |E_{X}(\Omega,d,u)|\right)^2\right)^{1/2}\nonumber\\
& \ll \Omega (\log N)^{11} \max_{X\in\{X_i,Y_j\}}\left(\sum_{d\leq Z} \frac{(\tau(d))^2}{d}\sum_{d\leq Z} \sum_{\substack{(u,d)=1\\ u=1}}^{d} |E_{X}(\Omega,d,u)|^2\right)^{1/2}.
\end{align}

From Lemma \ref{BVsquare}, we have 
\begin{align*}
&\sum_{d\leq x (\log x)^{-C}} \sum_{ \substack{(u,d)=1\\ u=1}}^d \left( \pi(x+\Omega,d,u)- \frac{\pi(x+\Omega)}{\phi(d)}-\pi(x,d,u)+ \frac{\pi(x)}{\phi(d)}\right)^2\\
& \ll \sum_{  d\leq x(\log x)^{-C}} \left\{\sum_{ \substack{ (u,d)=1\\ u=1}}^d \left( \pi (x+\Omega,d,u)- \frac{\pi(x+\Omega)}{\phi(d)}\right)^2+ \left( \pi(x,d,u)- \frac{\pi(x)}{\phi(d)}\right)^2 \right\}\\
& \ll (x+\Omega)^2 ( \log( x+\Omega))^{3-C}.
\end{align*}
Combining this with the fact that $\max_{i,j} \{X_i, Y_j\} \leq \sqrt{N}$, we see that 
\eqref{errorBV} becomes 
\begin{align}\label{boundBVaverage}
E'&\ll \Omega (\log N)^{11} \left(\sum_{d\leq Z} \frac{（(\tau(d))^2}{d} \sum_{d\leq Z} \sum_{ \substack{(u,d)=1\\ u=1}}^d \left( \pi(\sqrt{N}+\Omega,d,u)- \frac{\pi(\sqrt{N}+\Omega)}{\phi(d)}\right)^2 \right)^{1/2}\nonumber\\
& \ll \sqrt{N}(\log N)^{-5}(\log N)^{11} (\log N)^2 \sqrt{N}(\log N)^{2-A/2}\nonumber\\
& \ll N (\log N)^{10-A/2}.
\end{align}
Therefore, combining \eqref{boundary} and \eqref{boundBVaverage}, we have 
\begin{align*}
M_1&= \sum_{d\leq Z} \sum_{ \substack{u^2+v^2 \equiv -1 \pmod d\\ (uv,d)=1 \\ u,v\leq d}} \sum_{X_i, Y_j} \frac{\Delta_{X_i}(\Omega)}{\phi(d)} \frac{\Delta_{Y_j}(\Omega)}{\phi(d)}+ O( N (\log N)^{-3}) \\
& = \sum_{d\leq Z} \frac{1}{\phi(d)^2} \sum_{ \substack{ u^2+v^2 \equiv -1 \pmod d\\ (uv,d)=1 \\ u,v\leq d }} \sum_{ X_i,Y_j} \Delta_{X_i}(\Omega) \Delta_{Y_j}( \Omega)+O( N (\log N)^{-3})\\
& = \sum_{d\leq Z} \frac{1}{\phi(d)^2} \sum_{ \substack{ u^2 +v^2 \equiv -1 \pmod d\\ (uv,d)=1\\ u,v\leq d}} \left(\sum_{ p^2 +q^2 \leq N}1 + O \left( (\log N)^5 \left(\frac{\Omega}{d}\right)^2\right) \right) +O( N (\log N)^{-3})\\
& = \sum_{d\leq Z} \frac{s(d)}{\phi(d)^2}  \sum_{ p^2+q^2\leq N} 1 +O\left( \sum_{d\leq Z} \frac{r(d)}{\phi(d)^2} \frac{N(\log N)^{-5}}{d^2}\right)+ O( N (\log N)^{-3})\\
&= \sum_{d\leq Z} \frac{s(d)}{\phi(d)^2}\sum_{p^2+q^2\leq N} 1+ O \left(N (\log N)^{-3}\right),
\end{align*}
where $s(d)$ is defined in \eqref{sndef}.
Applying Lemma \ref{p2q2leqN} and Lemma \ref{AssympRd}, we have 
\begin{align}
M_1= \frac{\pi}{4}\prod_{p>2} \left( 1- \frac{1+3p\left(\frac{-1}{p}\right)}{(p-1)^2p}\right) \frac{N}{\log N} \left(1+O\left(\frac{(\log\log N)^2}{\log N}\right)\right).
\end{align}

\section{Estimation of $M_2$}\label{EvalM2}
Recall from \eqref{M2} that $M_2$ is defined by
\begin{align*}
M_2=2\sum_{Z<d\leq \sqrt{N+1}} \sum_{\substack{d^2-1 \leq p^2+q^2 \leq N\\ p^2+q^2\equiv -1\pmod d}} 1.
\end{align*}

Similarly to $M_1$, the terms in $M_2$ with $p<Z$ can be bounded by 
\begin{align*}
\sum_{\substack{p\leq Z, q\\ {p^2+q^2 \leq N} }} \sum_{\substack{Z< d\leq \sqrt{N+1}\\ d\mid p^2+q^2+1}}1
& \ll \left(\sum_{p\leq Z,q} 1 \right)^{1/2} \left( \sum_{p\leq Z,q \leq \sqrt{N}}\left(\sum_{\substack{Z<d\leq \sqrt{N+1}\\ d\mid p^2+q^2+1}} 1\right)^2\right)^{1/2}\\
& \ll \pi (Z)^{1/2}\pi(\sqrt{N})^{1/2} \left( \sum_{ n\leq N+1}(\tau(n))^2 \sum_{ \substack{p^2+q^2+1=n\\p\leq Z, q\leq \sqrt{N}}}1\right)^{1/2}\\
& \ll N (\log N)^{-A} \left( \sum_{n\leq N+1} (\tau(n))^2 r(n-1)\right)^{1/2}\\
& \ll N (\log N)^{-A} \left( \sum_{ n\leq N+1} (\tau(n))^2\tau(n-1)\right)^{1/2}\\
&\ll N( \log N)^{-A}\left( \sum_{ n\leq N} ( \tau(n))^4 \sum_{n\leq N+1}(\tau(n-1))^2\right)^{1/4}\\
& \ll N (\log N)^{-A/2+5}.
\end{align*}
The terms in $M_2$ with $p\mid d$ can be bounded by 
\begin{align*}
&\ll \sum_{Z<d\leq \sqrt{N+1}} \sum_{p\mid d} \sum_{ \substack{q\leq \sqrt{N}\\ q^2\equiv -1+p^2 \pmod d}}1\\
&\ll \sum_{2^k\leq \sqrt{N+1}} \sum_{\substack{Z\leq d\leq \sqrt{N+1}\\ 2\nmid d}} \sum_{p\mid d}\frac{\sqrt{N}}{d}\tau(d)\\
& \ll \sqrt{N}(\log N)\sum_{Z\leq d\leq \sqrt{N}} \frac{\tau(d)^2}{d}\\
& \ll \sqrt{N}(\log N)^5.
\end{align*}
Thus, 
\begin{align}\label{m2bounds}
M_2\ll\sum_{Z\leq d\leq \sqrt{N+1}} \sum_{\substack{ p^2+q^2\leq N\\ p^2+q^2 \equiv -1 \pmod d\\ (pq,d)=1\\ p\geq Z,q\geq Z}}1+ O(N(\log N)^{-A/2+5}).
\end{align}
In order to give an upper bound for $M_2$, we use upper bound sieve weights to detect the primality of $p$ and $q$.
First we recall the fundamental lemma of sieve theory.
\begin{lemma}[Fundamental lemma of sieve theory]\label{upperSieve}
	Let $y>1$ and $s\geq 1$. There exists a set of numbers $(\lambda_d)$ such that 
	\begin{enumerate}
		\item  $\lambda_1=1$
		\item $|\lambda_d|\leq 1$ if $1< d< y$.
		\item $\lambda_d=0$ if $d\geq y$.	
	\end{enumerate}
	and for any integer $n>1$, $0\leq \sum_{d\mid n} \lambda_d$. Moreover, for any multiplicative function $g(d)$ with $0\leq g(d)<1$ and satisfying the dimension condition 
	\begin{align}
	\prod_{w\leq p \leq z} \left(1-g(p)\right)^{-1}\leq \left( \frac{\log z}{\log w}\right)^{\kappa}\left(1+ \frac{K}{\log w}\right)
	\end{align} 
	for all $2\leq w < z\leq y$, we have
	\begin{align*}
	\sum_{d\mid P(z)}{\lambda_dg(d)}= \prod_{p<z}\left(1- {g(p)} \right)\left(1+ O\left(e^{-s}\frac{K}{\log z}\right)
	\right),
	\end{align*} where $P(z)=\prod_{p<z}p$ and $s=\log y/\log z$, the implied constant only depends on $\kappa$.
\end{lemma}
\begin{proof}
	See Lemma 6 in Chapter 6 of \cite{IwaniecBook}.
\end{proof}
Let 
 $\theta(m)=\sum_{\substack{e\mid n\\ e\leq E}} \lambda_e$, $E=N^\delta$, for some $0<\delta<1/2$. Let \begin{align}\label{defS}
 S=\sum_{Z\leq d\leq \sqrt{N+1}} \sum_{\substack{ m^2+n^2\leq N\\m^2+n^2 \equiv-1 \pmod d\\ (mn,d)=1 }}\theta (m)\theta(n)f(m)f(n),
 \end{align} where $f$ is a smooth function which is $1$ on $[\frac{Z}{2}, 2 \sqrt{N}]$. Since $\theta(p) \geq 1$ when $p >E$, thus $M_2\ll S$. From \eqref{m2bounds}, it is enough to obtain an upper bound for $S$. Suppose further that $f$ is bounded by $1$ elsewhere satisfying 
 \begin{align}
 f^{(n)}(x)\ll Z^{-n}\label{Fhat}
 \end{align} for all $n\geq 1$ and $x$. 
\begin{lemma}[Poisson Summation formula]\label{PossionSummation}Let $f: \RR\rightarrow \CC$ be a Schwartz function, i.e. $f$ is smooth and $|f(x)| \ll (1+|x|)^{-n}$ as $x\rightarrow\infty$ for all $n$. Then
	\begin{align*}
	\sum_{n=-\infty}^\infty f(t+nm)=\sum_{k=-\infty}^\infty \frac{1}{m} \hat{f}\left(\frac{k}{m}\right)e^{2\pi i \frac{kt}{m}},
	\end{align*}
	where $\hat{f}(k)=\int_{-\infty}^{\infty}f(x)e^{-2\pi ikx}dx$.
\end{lemma}
\begin{proof}
	See equation (4.24) in Chapter 4 of \cite{IwaniecBook}.
\end{proof}
We have
\begin{align}\label{fhat1}
\hat{f}(\lambda)=\int_\RR f(x)e(-\lambda x)dx\ll \sqrt{N}.
\end{align}
Also, from \eqref{Fhat}, 
\begin{align}\label{fhat2}
\hat{f}\left( \frac{h_1}{e_1d}\right) \ll \left(\frac{e_1d}{h_1}\right)^{j}Z^{-j}\sqrt{N}, \text{ for all }j\geq 1.
\end{align}
Applying Lemma \ref{PossionSummation}, we have 
\begin{align*}
S&=\sum_{e_1,e_2\leq E} \lambda_{e_1}\lambda_{e_2} \sum_{\substack{Z\leq d\leq \sqrt{N+1}\\ (e_1e_2,d)=1}} \sum_{\substack{e_1^2m^2+e_2^2n^2\equiv-1\pmod d\\ (mn,d)=1}} f(e_1m)f(e_2n)\\
&=\sum_{e_1,e_2\leq E} \lambda_{e_1}\lambda_{e_2} \sum_{\substack{Z\leq d\leq \sqrt{N+1}\\ (e_1e_2,d)=1}} \sum_{\substack{e_1^2u^2+e_2^2v^2\equiv-1\pmod d\\ (uv,d)=1\\ u,v\leq d}} \sum_{m\equiv u\pmod d}f(e_1m)\sum_{n\equiv v\pmod d}f(e_2n)\\
&= \sum_{e_1,e_2\leq E} \lambda_{e_1}\lambda_{e_2} \sum_{\substack{Z\leq d\leq \sqrt{N+1}\\ (e_1e_2,d)=1}}\frac{1}{d^2} \sum_{\substack{e_1^2u^2+e_2^2v^2\equiv-1\pmod d\\ (uv,d)=1}} \frac{1}{ e_1e_2} \sum_{h_1}\sum_{h_2} e\left( \frac{uh_1+vh_2}{d}\right)\hat{f}\left(\frac{h_1}{e_1 d}\right) \hat{f}\left( \frac{h_2}{e_2 d}\right).
\end{align*}
The terms with $h_1=h_2=0$ give a contribution of 
\begin{align}
&\sum_{e_1\leq E} \sum_{e_2\leq E}\lambda_{e_1}\lambda_{e_2} \sum_{\substack{Z\leq d\leq \sqrt{N+1}\\ (e_1e_2,d)=1}}\frac{1}{d^2e_1e_2} \sum_{ \substack{e_1^2u^2+e_2^2v^2\equiv -1\pmod d\\ (uv,d)=1}}\hat{f}(0)\hat{f}(0)\nonumber\\
& =\sum_{e_1, e_2 \leq E} \frac{\lambda_{e_1}\lambda_{e_2}}{e_1e_2} \sum_{\substack{Z\leq d\leq \sqrt{N+1}\\ (e_1e_2,d)=1}} \frac{r(d)}{d^2}(\hat{f}(0))^2\nonumber\\
&= (\hat{f}(0))^2\sum_{Z\leq d\leq \sqrt{N+1}} \frac{r(d)}{d^2} \sum_{\substack{e_1,e_2\leq E\\(e_1e_2,d)=1}} \frac{\lambda_{e_1}\lambda_{e_2}}{e_1e_2}\nonumber \\
&= (\hat{f}(0))^2\sum_{Z\leq d\leq \sqrt{N+1}} \frac{r(d)}{d^2} \left(\sum_{\substack{e\leq E\\(e,d)=1}} \frac{\lambda_{e}}{e}\right)^2.\label{h1h2=0}
\end{align}
Applying Lemma \ref{upperSieve} with $z=y=E$, we have 
\begin{align*}
\sum_{\substack{e_1\leq E\\(e_1,d)=1}} \frac{\lambda_{e_1}}{e_1}&\ll \prod_{\substack{p\leq E\\(p,d)=1}}\left(1-\frac{1}{p}\right)\\
& \ll \prod_{p\leq E} \left( 1- \frac{1}{p}\right) \prod_{ \substack{p\mid d\\ p\leq E}} \left(1- \frac{1}{p}\right)^{-1}\\
& \ll \prod_{p\leq E} \left( 1- \frac{1}{p}\right) \prod_{ \substack{p\mid d}} \left(1- \frac{1}{p}\right)^{-1}\\
&\ll \prod_{p\leq E} \left( 1- \frac{1}{p}\right) \frac{d}{\phi(d)}.
\end{align*}
From Lemma \ref{AssympRd}, we see that 
\begin{align*}
\sum_{Z\leq d\leq \sqrt{N+1}} \frac{s(d)}{\phi(d)^2} =\frac{1}{2}\prod_{\substack{p>2}} \left( 1+ \frac{1+3p \left( \frac{-1}{p}\right)}{(p-1)^2p}\right)\log \frac{\sqrt{N+1}}{Z}(1+o(1)).
\end{align*}
Since $E=N^\delta$, $Z=\frac{\sqrt{N}}{(\log N)^A}$, we see that \eqref{h1h2=0} is bounded from above by
\begin{align*}
\hat{f}(0)\hat{f}(0)\sum_{Z\leq d\leq \sqrt{N+1}} \frac{s(d)}{d^2} \prod_{p\leq E} \left(1- \frac{1}{p}\right)^2 \frac{d^2}{\phi(d)^2}
\ll & \frac{(\hat{f}(0))^2}{(\log E)^2}
 \log \frac{\sqrt{N+1}}{Z}\\
\ll &  N \frac{\log\log N}{(\log N)^2}.
\end{align*}

By breaking $e_1$, $e_2$ and $d$ into dyadic ranges , we need to consider 
\begin{align}\label{errorSumPartition}
\sum_{e_1 \sim E_1,e_2\sim E_2} \lambda_{e_1}\lambda_{e_2} \sum_{\substack{ d\sim D\\ (e_1e_2,d)=1}}\frac{1}{d^2} \sum_{\substack{e_1^2u^2+e_2^2v^2\equiv-1\pmod d\\ (uv,d)=1}} \frac{1}{ e_1e_2} \sum_{h_1}\sum_{h_2} e\left( \frac{uh_1+vh_2}{d}\right)\hat{f}\left(\frac{h_1}{e_1 d}\right) \hat{f}\left( \frac{h_2}{e_2 d}\right),
\end{align}
where $E_1,E_2\leq E$, $(h_1,h_2)\not=(0,0)$, and $Z\leq D\leq \sqrt{N+1}$. Since $E,D\ll N$, the number of $E_1,E_2$ and $D$ is bounded by $ N^{o(1)}$.
Applying \eqref{fhat2} with $j=n$ for $ \hat{f}\left( \frac{h_1}{e_1d}\right)$ and $j=2$ for $\hat{f}\left( \frac{h_2}{e_2d}\right)$, we see that the contribution from $|h_1|\geq \frac{D E_1 N^{\epsilon}}{\sqrt{N}}$ is bounded by
\begin{align*}
&\sum_{e_1\sim E_1,e_2\sim E_2} \lambda_{e_1}\lambda_{e_2} \sum_{\substack{d\sim D\\(e_1e_2,d)=1}} \frac{1}{d^2} \sum_{ \substack{e_1^2 u^2+e_2^2v^2 \equiv -1 \pmod d\\ (uv,d)=1}} \frac{1}{e_1e_2} \sum_{|h_1| \geq \frac{DE_1 N^{\epsilon}}{\sqrt{N}}} \sum_{h_2}\hat{f}\left( \frac{h_1}{e_1d}\right) \hat{f}\left( \frac{h_2}{e_2d}\right)\\
& \ll \sum_{e_1\sim E_1,e_2\sim E_2}\frac{1}{e_1e_2} \sum_{ d\sim D } \frac{r(d)}{d^2}  \sum_{|h_1|\geq \frac{ DE_1 N^{\epsilon}}{\sqrt{N}}}\left(\frac{e_1d}{h_1}\right)^{n} \left(\sum_{h_2\not=0} \left( \frac{e_2d}{h_2}\right)^{2}+\hat{f}(0)\right)\\
& \ll N^\epsilon (E_1D)^{n}\left(\frac{\sqrt{N}}{E_1DN^{\epsilon}}\right)^{n-1}Z^{-n}\sqrt{N}\left((E_2D)^2Z^{-2}\sqrt{N}+\sqrt{N}\right)\\
& \ll N^\epsilon E_1E_2^2D^3 N^{-\epsilon n/2-1/2}+N^\epsilon E_1 DN^{-\epsilon n/2+1/2}\\
& \ll N^{-\delta},
\end{align*}
by taking $n$ sufficiently large. The terms with $|h_2| \geq \frac{DE_2N^\epsilon}{\sqrt{N}}$ can be bounded $N^{-\delta}$ in the same way. Thus it remains to consider the case $0\leq h_1\leq \frac{DE_1 N^\epsilon}{\sqrt{N}}$, $0\leq h_2 \leq \frac{DE_2 N^\epsilon}{\sqrt{N}}$ and $(h_1,h_2)\not=(0,0)$. Denote
\begin{align}\label{Edef}
E(e_1,e_2,h_1,h_2,d)=\sum_{ \substack{e_1^2 u^2 +e_2^2v^2 \equiv -1 \pmod d\\ (uv,d)=1}} e\left( \frac{uh_1+vh_2}{d}\right).
\end{align}
 We use the following lemma to complete the estimates for $M_2$, and the proof of Lemma \ref{ExpSumEst} is given in Section \ref{ExpSumEstProof}. 
 \begin{lemma}\label{ExpSumEst}If $(h_1,h_2)\not=(0,0)$, then
 	\begin{align*}
 	E(e_1,e_2,h_1,h_2,d) \ll C^{\omega(d)}\sqrt{(h_1,h_2,d)d},
 	\end{align*} where $C>0$ is an absolute constant. 
 \end{lemma}
 Applying Lemma \ref{ExpSumEst} to \eqref{errorSumPartition}, we have 
 \begin{align*}
 &\sum_{e_1 \sim E_1,e_2\sim E_2} \lambda_{e_1}\lambda_{e_2} \sum_{\substack{ d\sim D\\ (e_1e_2,d)=1}}\frac{1}{d^2} \frac{1}{ e_1e_2} \sum_{|h_1|\leq \frac{DE_1 N^\epsilon}{\sqrt{N}}}\sum_{\substack{|h_2|\leq \frac{DE_2 N^\epsilon}{\sqrt{N}}\\ (h_1,h_2)\not=(0,0)}} \hat{f}\left(\frac{h_1}{e_1 d}\right) \hat{f}\left( \frac{h_2}{e_2 d}\right)E(e_1,e_2,h_1,h_2,d)\\
 & \ll N \sum_{e_1 \sim E_1} \sum_{e_2 \sim E_2}  \frac{1}{e_1e_2}\sum_{\substack{d\sim D\\ (e_1e_2,d)=1}} \frac{1}{d^2} \sum_{ |h_1|\frac{\leq DE_1 N^\epsilon}{\sqrt{N}}} \sum_{ |h_2| \leq \frac{DE_2 N^\epsilon}{\sqrt{N}}} C^{\omega(d)}\sqrt{(h_1,h_2,d) d}\\
 & \ll N^{1+\epsilon} \sum_{ g\leq D} \frac{C^{\omega(g)}}{g^2}\sum_{d\sim D/g} \frac{1}{d^2} \sum_{ |h_1| \leq \frac{DE_1N^\epsilon}{\sqrt{N}g}} \sum_{|h_2|\leq \frac{DE_2N^\epsilon}{\sqrt{N}g}} C^{\omega(d)} \sqrt{g gd}\\
 & \ll N^{1+\epsilon} \sum_{g\leq D} \frac{C^{\omega(g)}}{g} \frac{g}{D} \frac{DE_1N^\epsilon}{\sqrt{N}g} \frac{D E_2 N^\epsilon}{\sqrt{N}g}\max_{d\sim D}C^{\omega(d)}\sqrt{d}\\
 & \ll \sum_{ g\leq D} \frac{\tau(g)^{\log C/\log 2}}{g}D E_1 E_2N^\epsilon \max_{d\sim D }\tau(d)^{\log C/\log 2}\sqrt{d}\\
 & \ll D^{3/2+\epsilon}E_1E_2 N^\epsilon.
 \end{align*}
 Choosing $E\ll N^{1/8-\delta_0}$, we find that $S\ll N^{1-\delta'}$ for some $\delta'>0$.
 \section{Proof of Lemma \ref{ExpSumEst}}\label{ExpSumEstProof}
\subsection{Quadratic Gauss Sums and Twisted Kloosterman Sums}
\subsubsection{Quadratic Gauss Sum}
Let $a,b,d$ be natural numbers. The quadratic Gauss sum is defined by 
\begin{align}\label{defGaussSum}
S(a,b,d):= \sum_{ n\pmod d} e\left( \frac{an^2+bn}{d}\right).
\end{align}
\begin{lemma}
	 We have the following properties of $S(a,b,d)$.
	 \begin{enumerate}
		\item If $(c,d)=1$, then $S(a,b,cd)=S(ac,b,d)S(ad,b,c)$.
		\item 
		If $(a,d)>1$, then $S(a,b,d)=0$ except when $(a,d)\mid b$, then 
		\begin{align}\label{factor}
		S(a,b,d)= (a,d) S\left(\frac{a}{(a,d)}, \frac{b}{(a,d)}, \frac{d}{(a,d)}\right).
		\end{align}
		\item For $(a,p)=1$ and $p>2$, 
		\begin{align}\label{completeSq}
		S(a,b,p^\alpha)=\sum_{n\pmod {p^\alpha}} e\left( \frac{an^2+bn}{p^\alpha}\right)= \left( \frac{a}{p^\alpha}\right) S(1,0,p^\alpha) e\left( -\frac{\overline{4a}b^2}{p^\alpha}\right)
		\end{align}
		\item \begin{align}
		S(1,0,p^\alpha)
		&=p S(1,0,p^{\alpha-2}), \alpha>2\label{GaussSumRelation}\\
		S(1,0,p^2)&=p.\label{SpSquare}
		\end{align}
		\item \begin{align}\label{GaussSumEval}
		S(1,0,d)=\sqrt{d^*}
		\end{align}
	\end{enumerate}
\end{lemma}
\begin{proof}
	See Chapter 3 of \cite{IwaniecBook}.
\end{proof}
%
%
\subsubsection{Kloosterman Sums}
Let $a,b,m$ be natural numbers. The Kloosterman sum is defined by
\begin{align}\label{KloostermanSumDef}
K(a,b;m)=\sum_{ \substack{(x,m)=1\\ x\pmod m}} e\left( \frac{ax+b \overline{x}}{m}\right),
\end{align}
where $\overline{x}$ is the inverse of $x$ modulo $m$.
\begin{lemma}\label{KloostermanSumBounds}
Let $K(a,b;m)$ be defined as above. Then
	\begin{align*}
	|K(a,b;m)|\leq \tau(m)\sqrt{(a,b,m)}\sqrt{m}.
	\end{align*}
\end{lemma}
\begin{proof}
	See corollary 11.12 in chapter 11 of \cite{IwaniecBook}.
\end{proof}
\subsubsection{Sali\'e sums}
Let $m,n,d$ be natural numbers. The Sale\'e sum is defined by
\begin{align*}
	T(m,n;d):= \sum_{x\pmod {d} } \left( \frac{x}{d}\right) e\left( \frac{m \bar{x}+ n x}{d}\right),
\end{align*}
where $\left( \frac{\cdot}{d}\right)$ is the Jacobi-Legendre symbol.

\begin{lemma}\label{Salie}
	Suppose $(d,2mn)=1$, Then $T(m,n,d)$ vanishes unless there exists an $a$ with $a^2\equiv mn \pmod { p^\beta}$. Given $a$, all the solutions to $x^2 \equiv mn \pmod { d}$ can be written explicitly as $x= (r\bar{r}-s\bar{s}) a$, where $r,s$ run over the factorizations of $rs=d$ with $(r,s)=1$. 
	$$T(m,n;d)=\sqrt{d^*} \left( \frac{n}{d}\right) \sum_{\substack{rs=d\\(r,s)=1}} e\left( 2a\left( \frac{\bar{r}}{s}-\frac{\bar{s}}{r}\right)\right).$$
 
\end{lemma}
\begin{proof}
	See equation (12.43) in Chapter 12 of \cite{IwaniecBook}.
\end{proof} 
As a corollary of Lemma \ref{Salie}, we see that 
\begin{corollary} \label{SalieSum}Let $T(m,n;d)$ be as above. Then,
$$T(m,n;d) \ll \sqrt{d}2^{\omega(d)}.$$
\end{corollary}
\begin{lemma}\label{ExpSum*}
	Let $\ell $ be a prime and $k\geq 1$ be an integer. Then,
	\begin{align*}
	\sum_{ \substack{(a,\ell)=1\\ a\pmod{ \ell^k}}} e\left( \frac{a}{\ell^k}\right)=  \left\{\begin{array}{ll}
	-1, & k=1,\\
	0, &  k\geq 2.
	\end{array}\right.
	\end{align*}
\end{lemma}
	\begin{proof}
	\begin{align*}
	\sum_{ \substack{(a,\ell)=1\\ a\pmod { \ell^k}}} e\left( \frac{a}{\ell^k}\right)= \sum_{a\pmod {\ell^k}} e\left( \frac{a}{\ell^k}\right)- \sum_{ a\pmod {\ell^{k-1}}} e\left( \frac{a}{\ell^{k-1}}\right)= \left\{\begin{array}{ll}
-1, & k=1,\\
0, &  k\geq 2.
	\end{array}\right.
	\end{align*}
	\end{proof}

Now we are ready to prove Lemma \ref{ExpSumEst}. 
\begin{proof}[Proof of Lemma \ref{ExpSumEst}] We rewrite \eqref{Edef} as
	\begin{align*}
	E(e_1,e_2,h_1,h_2,d)&=\sum_{ \substack{e_1^2 u^2+e_2^2v^2\equiv -1\pmod d\\ (uv,d)=1}} e\left( \frac{uh_1+vh_2}{d}\right)\\
	&=\frac{1}{d} \sum_{a\pmod d} \sum_{\substack{u \pmod d\\ (u,d)=1}} \sum_{\substack{v\pmod d\\ (v,d)=1}}e\left( \frac{uh_1+vh_2}{d}\right) e\left( \frac{a(e_1^2u^2+e_2^2v^2+1)}{d}\right)\\
	&= \frac{1}{d}\sum_{a\pmod d}e\left( \frac{a}{d}\right)\sum_{\substack{u\pmod d\\ (u,d)=1}}e\left(\frac{ae_1^2u^2+uh_1}{d}\right)\sum_{\substack{v\pmod d\\ (v,d)=1}}e\left( \frac{ae_2^2v^2+vh_2}{d}\right).
	\end{align*}
	From the Chinese remainder theorem, it is enough to consider $E(e_1,e_2,h_1,h_2,\ell^\alpha)$ for primes $\ell$.
	For $(e_1e_2,\ell)=1$, we have
	\begin{align}
	&E(e_1,e_2,h_1,h_2,\ell^\alpha)\nonumber\\&= \frac{1}{\ell^\alpha} \sum_{a\pmod {\ell^\alpha}} \sum_{ (uv,\ell)=1} e\left( \frac{h_1\overline{e_1}u+h_2\overline{e_2}v}{\ell^\alpha}\right)e\left( \frac{au^2+av^2+a}{\ell^\alpha}\right)\nonumber\\
	&=\frac{1}{\ell^\alpha}\sum_{k=1}^\alpha \sum_{a=\ell^k}e\left( \frac{\ell^{\alpha-k}a}{\ell^\alpha}\right) \sum_{ \substack{(u,\ell)=1\\ u \pmod {\ell^\alpha}}}  e\left( \frac{\ell^{\alpha-k}au^2+h_1\overline{e_1}u}{\ell^\alpha}\right)\sum_{ \substack{(v,\ell)=1\\ v \pmod {\ell^\alpha}}}  e\left( \frac{\ell^{\alpha-k}av^2+h_2\overline{e_2}v}{\ell^\alpha}\right) \nonumber\\&+\frac{1}{\ell^\alpha} \sum_{k=1}^\alpha \sum_{ \substack{(a,\ell)=1\\ a\pmod{\ell^k}}}e\left( \frac{\ell^{\alpha-k}a}{\ell^\alpha}\right) \sum_{ \substack{(u,\ell)=1\\ u \pmod {\ell^\alpha}}}  e\left( \frac{\ell^{\alpha-k}au^2+h_1\overline{e_1}u}{\ell^\alpha}\right)\sum_{ \substack{(v,\ell)=1\\ v \pmod {\ell^\alpha}}}  e\left( \frac{\ell^{\alpha-k}av^2+h_2\overline{e_2}v}{\ell^\alpha}\right).\label{sumCoprime}
	\end{align}
	From Lemma \ref{ExpSum*}, we see that 
	\begin{align*}
\frac{1}{\ell^\alpha}\sum_{k=1}^\alpha \sum_{a=\ell^k}e\left( \frac{\ell^{\alpha-k}a}{\ell^\alpha}\right) \sum_{ \substack{(u,\ell)=1\\ u \pmod {\ell^\alpha}}}  e\left( \frac{\ell^{\alpha-k}au^2+h_1\overline{e_1}u}{\ell^\alpha}\right)\sum_{ \substack{(v,\ell)=1\\ v \pmod {\ell^\alpha}}}  e\left( \frac{\ell^{\alpha-k}av^2+h_2\overline{e_2}v}{\ell^\alpha}\right)= \frac{1}{\ell^\alpha}.
	\end{align*}
	
	For $(a,\ell)=1$, $\ell^{\alpha-k+1}\mid h_1$, from \eqref{factor}, \eqref{completeSq}, and \eqref{GaussSumRelation}, after writing $h_1=\ell^{\alpha-k+1}h_1'$, we have that if $k\geq 3$, 
	\begin{align}
	&\sum_{ \substack{(u,\ell)=1\nonumber\\ u \pmod {\ell^\alpha}}}  e\left( \frac{\ell^{\alpha-k}au^2+h_1\overline{e_1}u}{\ell^\alpha}\right)\\&= \sum_{ u\pmod {\ell^\alpha}} e\left( \frac{\ell^{\alpha-k}au^2+h_1 \overline{e_1}u}{\ell^\alpha}\right)- \sum_{ u \pmod {\ell^{\alpha-1}}} e\left( \frac{\ell^{\alpha-k+1}au^2+h_1\overline{e_1}u}{\ell^{\alpha-1}}\right)\nonumber\\
	&= \ell^{\alpha-k} \sum_{u \pmod { \ell^k}}e\left( \frac{au^2+h_1'\ell\overline{e_1}u}{\ell^k}\right)-\ell^{\alpha-k+1} \sum_{u \pmod { \ell^{k-2}}}e\left( \frac{au^2+h_1'\overline{e_1}u}{\ell^{k-2}}\right)\nonumber\\
	&= \ell^{\alpha-k} \left( \frac{a}{\ell^k}\right)e\left( \frac{-\overline{4ae_1^2} h_1'^2\ell^2}{\ell^k}\right)S(1,0,\ell^k)- \ell^{\alpha-k+1} \left( \frac{a}{\ell^{k-2}}\right)e\left( \frac{-\overline{4ae_1^2} h_1'^2}{\ell^{k-2}}\right)S(1,0,\ell^{k-2})\nonumber\\
	&=0.\label{k>2}
	\end{align}
		For $(a,\ell)=1$, $\ell^{\alpha-k+1}\mid h_1$, from \eqref{factor}, \eqref{completeSq}, and \eqref{GaussSumRelation}, after writing $h_1=\ell^{\alpha-k+1}h_1'$, we have that if $k< 3$, then $\ell^{\alpha-1}\mid h_1$. It thus follows that 
		\begin{align}
		&\sum_{ \substack{(u,\ell)=1\nonumber\\ u \pmod {\ell^\alpha}}}  e\left( \frac{\ell^{\alpha-k}au^2+h_1\overline{e_1}u}{\ell^\alpha}\right)\\&= \sum_{ u\pmod {\ell^\alpha}} e\left( \frac{\ell^{\alpha-k}au^2+h_1 \overline{e_1}u}{\ell^\alpha}\right)- \sum_{ u \pmod {\ell^{\alpha-1}}} e\left( \frac{\ell^{\alpha-k+1}au^2+h_1\overline{e_1}u}{\ell^{\alpha-1}}\right)\nonumber\\
		&= \ell^{\alpha-k} \sum_{u \pmod { \ell^k}}e\left( \frac{au^2+h_1'\ell\overline{e_1}u}{\ell^k}\right)-\sum_{u \pmod { \ell^{\alpha-1}}}e\left( \frac{h_1\overline{e_1}u}{\ell^{\alpha-1}}\right)\nonumber\\
		&= \ell^{\alpha-k} \left( \frac{a}{\ell^k}\right)e\left( \frac{-\overline{4ae_1^2} h_1'^2\ell^2}{\ell^k}\right)S(1,0,\ell^k)-\ell^{\alpha-1}\nonumber\\
		&= \ell^{\alpha-k} \left( \frac{a}{\ell^k}\right)S(1,0,\ell^k)-\ell^{\alpha-1}.\label{k<3}
		\end{align}	
	Similarly, for $(a,\ell)=1$, $\ell^{\alpha-k}\mid\mid h_1$, after writing $h_1=\ell^{\alpha-k}h_1'$, we have that if $k\geq2$, then
	\begin{align}
	&\sum_{ \substack{(u,\ell)=1\\ u \pmod {\ell^\alpha}}}  e\left( \frac{\ell^{\alpha-k}au^2+h_1\overline{e_1}u}{\ell^\alpha}\right)= \ell^{\alpha-k} \left( \frac{a}{\ell^k}\right)e\left( \frac{-\overline{4ae_1^2} h_1'^2}{\ell^k}\right)S(1,0,\ell^k),\label{exactDivideh}
	\end{align}
	and if $k=1$, then 
		\begin{align}
		&\sum_{ \substack{(u,\ell)=1\\ u \pmod {\ell^\alpha}}}  e\left( \frac{\ell^{\alpha-1}au^2+h_1\overline{e_1}u}{\ell^\alpha}\right)= \ell^{\alpha-1} \left( \frac{a}{\ell}\right)e\left( \frac{-\overline{4ae_1^2} h_1'^2}{\ell}\right)S(1,0,\ell)-\ell^{\alpha-1}.\label{exactDividehk=1}
		\end{align}
For $(a,\ell)=1$, $\ell^{\alpha-k}\nmid h_1$, we have that if $k\geq 2$,
	\begin{align}
	&\sum_{ \substack{(u,\ell)=1\\ u \pmod {\ell^\alpha}}}  e\left( \frac{\ell^{\alpha-k}au^2+h_1\overline{e_1}u}{\ell^\alpha}\right)=0.\label{hcoprime}
	\end{align}
	and that if $k=1$, 	\begin{align}
	&\sum_{ \substack{(u,\ell)=1\\ u \pmod {\ell^\alpha}}}  e\left( \frac{\ell^{\alpha-k}au^2+h_1\overline{e_1}u}{\ell^\alpha}\right)=-\sum_{ u \pmod{\ell^{\alpha-1}} }e\left( \frac{h_1\overline{e_1}u}{\ell^{\alpha-1}}\right)=\left\{\begin{array}{ll}
-1, & \alpha=1,\\
0, & \alpha\geq 2.
	\end{array}.\right.\label{hcoprimek=1}
	\end{align}
	Let $h_1=\ell^{t}h_1'$ and $h_2=\ell^sh_2'$, where $(h_1'h_2',\ell)=1$. From \eqref{hcoprime} and \eqref{hcoprimek=1}, we see that only the terms with $k$ satisfying $\alpha-k \leq t$ and $\alpha-k \leq s$ will contribute to the sum \eqref{sumCoprime} unless $\alpha=1$. Without loss of generality, we can assume $t\leq s$. Thus we only need to consider $k\geq \alpha-t\geq \alpha-s$ when $\alpha\geq 2$.
From \eqref{k>2}, \eqref{k<3} and \eqref{exactDivideh}, we see that we can further restrict $k$ such that $k=1,2,\alpha-t$. In the following we consider $\alpha=1$ in Case 0 and  $\alpha\geq 2$ in Case 1-Case 6.

Case 0. For prime $\ell$, $(e_1e_2,\ell)=1$, we have 
\begin{align}
&\sum_{ \substack{e_1^2 u^2+e_2^2 v^2 \equiv -1\pmod \ell\\ (uv,\ell)=1}} e\left( \frac{h_1 u+ h_2 v}{\ell }\right)\nonumber\\
&= \sum_{ \substack{u^2+v^2\equiv -1 \pmod \ell \\ (uv,\ell)=1}} e\left( \frac{h_1 \overline{e_1}u+ h_2 \overline{e_2}v}{\ell}\right)\nonumber\\
&= \frac{1}{\ell }\sum_{a\mod \ell} \sum_{ (uv,\ell)=1} e\left(  \frac{h_1 \overline{e_1}u+ h_2 \overline{e_2}v}{\ell}\right) e\left( \frac{a( u^2+v^2+1)}{\ell}\right)\nonumber\\
& = \frac{1}{\ell }+ \frac{1}{\ell }\sum_{(a,\ell)=1} \sum_{ (uv,\ell)=1} e\left(  \frac{h_1 \overline{e_1}u+ h_2 \overline{e_2}v}{\ell}\right) e\left( \frac{a( u^2+v^2+1)}{\ell}\right)\nonumber\\
& = \frac{1}{\ell }+ \frac{1}{\ell} \sum_{ (a,\ell)=1} e\left( \frac{a}{\ell}\right) \sum_{ (u,\ell)=1}e\left( \frac{au^2+h_1 \overline{e_1}u}{\ell}\right) \sum_{ (v,\ell)=1}e\left( \frac{av^2+h_2 \overline{e_2}v}{\ell}\right)\nonumber\\
&= \frac{1}{\ell} + \frac{1}{\ell}\sum_{(a,\ell)=1} e\left( \frac{a}{\ell}\right) \left(\left( \frac{a}{\ell }\right) e\left( \frac{-\overline{4a}\overline{e_1}^2h_1^2}{\ell}\right)\sqrt{\ell^*}-1\right)\left(\left( \frac{a}{\ell }\right) e\left( \frac{-\overline{4a}\overline{e_2}^2h_2^2}{\ell}\right)\sqrt{\ell^*}-1\right)\nonumber\\
&= \frac{1}{\ell}+ \sum_{ (a,\ell)=1}e\left( \frac{a-\overline{4ae_1^2}h_1^2-\overline{4ae_2^2}h_2}{\ell}\right)\left( \frac{-1}{\ell}\right) + O \left( \sqrt{\ell}\right)\nonumber\\
&= O \left( \sqrt{\ell}\right).\label{primeModuli}
\end{align}

Case 1.
If $t<\alpha-1$, then $\ell^{\alpha-1}\nmid h_1$, thus only terms with $k=\alpha-t\geq 2$ contribute to \eqref{sumCoprime} when $\alpha\geq 2$ by \eqref{hcoprime} and \eqref{hcoprimek=1}. If $t=s<\alpha-1$, then we have 
\begin{align*}
&E(e_1,e_2,h_1,h_2,\ell^\alpha)\\
&= \frac{1}{\ell^\alpha}+\frac{1}{\ell^\alpha} \sum_{k=1,2\alpha-t} \sum_{ \substack{(a,\ell)=1\\ a\pmod{\ell^k}}}e\left( \frac{\ell^{\alpha-k}a}{\ell^\alpha}\right) \sum_{ \substack{(u,\ell)=1\\ u \pmod {\ell^\alpha}}}  e\left( \frac{\ell^{\alpha-k}au^2+h_1\overline{e_1}u}{\ell^\alpha}\right)\sum_{ \substack{(v,\ell)=1\\ v \pmod {\ell^\alpha}}}  e\left( \frac{\ell^{\alpha-k}av^2+h_2\overline{e_2}v}{\ell^\alpha}\right)\\
&= \frac{1}{\ell^\alpha}+ \frac{1}{\ell^\alpha} \sum_{k=\alpha-t} \sum_{ \substack{(a,\ell)=1\\ a\pmod{\ell^k}}}e\left( \frac{a}{\ell^k}\right) \ell^{\alpha-k} e\left( \frac{-\overline{4ae_1^2}h_1'^2}{\ell^k}\right)S(1,0,\ell^k)\ell^{\alpha-k} e\left( \frac{-\overline{4ae_2^2}h_2'^2}{\ell^k}\right)S(1,0,\ell^k)\\
&= \frac{1}{\ell^\alpha}+ \ell^{t} \left( \frac{-1}{\ell^{\alpha-t}}\right) \sum_{ \substack{(a,\ell)=1\\ a\pmod { \ell^{\alpha-t}}}} e\left( \frac{a-\bar{a}(\overline{4e_1^2}h_1'^2+4\overline{e_2^2}h_2'^2)}{\ell^{\alpha-t}}\right) \\
&= O \left( \sqrt{\ell^{\alpha+t}}\right),
\end{align*}
where the last equality follows from Lemma \ref{KloostermanSumBounds}.  

Case 2. If $s\geq \alpha-1>t$, then from \eqref{k>2}, we see that if  $k=\alpha-t\geq3$ then $$E(e_1,e_2,h_1,h_2, \ell^\alpha)=0=O\left( \sqrt{\ell^{\alpha+t}}\right).$$ 

Case 3. When $s\geq \alpha-1>t,k=\alpha-t=2$, from \eqref{k<3} we have 
\begin{align*}
&E(e_1,e_2,h_1,h_2,\ell^\alpha)\\
&= \frac{1}{\ell^\alpha}+\frac{1}{\ell^\alpha} \sum_{k=1,2,\alpha-t} \sum_{ \substack{(a,\ell)=1\\ a\pmod{\ell^k}}}e\left( \frac{\ell^{\alpha-k}a}{\ell^\alpha}\right) \sum_{ \substack{(u,\ell)=1\\ u \pmod {\ell^\alpha}}}  e\left( \frac{\ell^{\alpha-k}au^2+h_1\overline{e_1}u}{\ell^\alpha}\right)\sum_{ \substack{(v,\ell)=1\\ v \pmod {\ell^\alpha}}}  e\left( \frac{\ell^{\alpha-k}av^2+h_2\overline{e_2}v}{\ell^\alpha}\right)\\
&= \frac{1}{\ell^\alpha}+ \frac{1}{\ell^\alpha} \sum_{k=\alpha-t} \sum_{ \substack{(a,\ell)=1\\ a\pmod{\ell^k}}}e\left( \frac{a}{\ell^k}\right) \ell^{\alpha-k} e\left( \frac{-\overline{4ae_1^2}h_1'^2}{\ell^k}\right)S(1,0,\ell^k)\left(\ell^{\alpha-k} \left( \frac{a}{\ell^k}\right)S(1,0,\ell^k)-\ell^{\alpha-1}\right)\\
&= \frac{1}{\ell^\alpha}+ \ell^{t} \left( \frac{-1}{\ell^{\alpha-t}}\right) \sum_{ \substack{(a,\ell)=1\\ a\pmod { \ell^{\alpha-t}}}}\left( \frac{a}{\ell^{\alpha-t}}\right) e\left( \frac{a-\bar{a}(\overline{4e_1^2}h_1'^2)}{\ell^{\alpha-t}}\right)- \frac{\ell^{\alpha-k+\alpha-1}}{\ell^\alpha}\sum_{\substack{(a,\ell)=1\\ a \pmod{\ell^k}}} e\left( \frac{a-\overline{a}(\overline{4e_1^2}h_1'^2)}{\ell^k}\right)S(1,0,\ell^k)\\
&=O \left( \sqrt{\ell^{\alpha+t}}\right),
\end{align*}
where we used Lemma \ref{SalieSum}.
 
Case 4. If $s>t=\alpha-1$, then we have 
\begin{align*}
E&(e_1,e_2,h_1,h_2,\ell^\alpha)\\
=& \frac{1}{\ell^\alpha}+ \frac{1}{\ell^\alpha}  \sum_{ \substack{(a,\ell)=1\\ a\pmod{\ell}}}e\left( \frac{a}{\ell}\right)\left( \ell^{\alpha-1} \left( \frac{a}{\ell}\right)e\left( \frac{-\overline{4ae_1^2} h_1'^2}{\ell}\right)S(1,0,\ell)-\ell^{\alpha-1}\right)\left( \ell^{\alpha-1} \left( \frac{a}{\ell}\right)S(1,0,\ell)-\ell^{\alpha-1}\right)\\
&+  \frac{1}{\ell^\alpha}  \sum_{ \substack{(a,\ell)=1\\ a\pmod{\ell^2}}}e\left( \frac{a}{\ell^2}\right)\left( \ell^{\alpha-2} \left( \frac{a}{\ell}\right)S(1,0,\ell^2)-\ell^{\alpha-1}\right)\left( \ell^{\alpha-2} \left( \frac{a}{\ell}\right)S(1,0,\ell^2)-\ell^{\alpha-1}\right)\\
=& \frac{1}{\ell^\alpha}+ \ell^{\alpha-1} \left( \frac{-1}{\ell}\right) \sum_{ \substack{(a,\ell)=1\\ a\pmod { \ell}}} e\left( \frac{a-\bar{a}\overline{4e_1^2}h_1'^2}{\ell}\right)+ 2\ell^{\alpha-2}\\&-\ell^{\alpha-2}\sum_{ \substack{(a,\ell)=1\\ a\pmod { \ell}}}\left( \frac{a}{\ell}\right)\left(e\left( \frac{a-\overline{4e_1^2}h_1'^2\bar{a}}{\ell}\right)+ e\left( \frac{a}{\ell}\right)\right)S(1,0,\ell) \\\\&-2\ell^{\alpha-3}\sum_{ \substack{(a,\ell)=1\\ a\pmod { \ell^2}}}e\left( \frac{a}{\ell^2}\right)\left( \frac{a}{\ell}\right)S(1,0,\ell^2)+ \sum_{ \substack{(a,\ell)=1\\ a\pmod { \ell^2}}} e\left( \frac{a}{\ell^2}\right) S(1,0,\ell^2)^2\ell^{\alpha-4}\\
=& O \left( \sqrt{\ell^{\alpha+t}}\right).
\end{align*}  

Case 5. If $s\geq t\geq \alpha$, then from \eqref{k<3}, we have 
\begin{align*}
&E(e_1,e_2,h_1,h_2,\ell^\alpha)\\&= \frac{1}{\ell^\alpha}+ \frac{1}{\ell^\alpha} \sum_{k=1}^2 \sum_{ \substack{(a,\ell)=1\\ a\pmod {\ell^k}}} e\left( \frac{a}{\ell^k}\right) \left( \ell^{\alpha-k} \left( \frac{a}{\ell^k}\right)S(1,0,\ell^k)-\ell^{\alpha-1}\right)\left(\ell^{\alpha-k} \left( \frac{a}{\ell^k}\right)S(1,0,\ell^k)-\ell^{\alpha-1}\right)\\
&= \frac{1}{\ell^\alpha}+ \sum_{k=1}^2 \sum_{ \substack{(a,\ell)=1\\ a\pmod {\ell^k}}} e\left( \frac{a}{\ell^k}\right)\left( \ell^{\alpha-k} \left( \frac{-1}{\ell^k}\right)-2\ell^{\alpha-k-1}S(1,0,\ell^k)\right)+ \ell^{\alpha-2}\\
&=O \left( \ell^{\alpha-1}\right)= O \left( \sqrt{\ell^{2\alpha}}\right).
\end{align*}
where the last equality follows from Lemma \eqref{ExpSum*} for $k\leq 2$. 

Case 6. If $s=t= \alpha-1$, then $k=1,2$ contribute to \eqref{sumCoprime}. From \eqref{exactDividehk=1}, \eqref{k<3} and Lemma \ref{ExpSum*}, Lemma \ref{Salie}, we have
\begin{align*}
&E(e_1,e_2,h_1,h_2,\ell^\alpha)\\
&= \frac{1}{\ell^\alpha}+ \frac{1}{\ell^\alpha}  \sum_{ \substack{(a,\ell)=1\\ a\pmod{\ell}}}e\left( \frac{a}{\ell}\right)\left( \ell^{\alpha-1} \left( \frac{a}{\ell}\right)e\left( \frac{-\overline{4ae_1^2} h_1'^2}{\ell}\right)S(1,0,\ell)-\ell^{\alpha-1}\right)\\ & \ \ \ \ \ \ \ \ \ \ \ \ \ \ \ \ \ \ \ \ \times\left( \ell^{\alpha-1} \left( \frac{a}{\ell}\right)e\left( \frac{-\overline{4ae_1^2} h_2'^2}{\ell}\right)S(1,0,\ell)-\ell^{\alpha-1}\right)\\
&+  \frac{1}{\ell^\alpha}  \sum_{ \substack{(a,\ell)=1\\ a\pmod{\ell^2}}}e\left( \frac{a}{\ell^2}\right)\left( \ell^{\alpha-2} \left( \frac{a}{\ell}\right)S(1,0,\ell^2)-\ell^{\alpha-1}\right)\left( \ell^{\alpha-2} \left( \frac{a}{\ell}\right)S(1,0,\ell^2)-\ell^{\alpha-1}\right)\\
&= \frac{1}{\ell^\alpha}+ \ell^{\alpha-1} \left( \frac{-1}{\ell}\right) \sum_{ \substack{(a,\ell)=1\\ a\pmod { \ell}}} e\left( \frac{a-\bar{a}(\overline{4e_1^2}h_1'^2+4\overline{e_2^2}h_2'^2)}{\ell}\right)+ 2\ell^{\alpha-2}\\& \ \ -\ell^{\alpha-2}\sum_{ \substack{(a,\ell)=1\\ a\pmod { \ell}}}\left( \frac{a}{\ell}\right)\left(e\left( \frac{a-\overline{4e_1^2}h_1'^2\bar{a}}{\ell}\right)+ e\left( \frac{a-\overline{4e_2^2}h_2'^2\bar{a}}{\ell}\right)\right)S(1,0,\ell) \\\\& \ \ -2\ell^{\alpha-3}\sum_{ \substack{(a,\ell)=1\\ a\pmod { \ell^2}}}e\left( \frac{a}{\ell^2}\right)\left( \frac{a}{\ell}\right)S(1,0,\ell^2)+ \sum_{ \substack{(a,\ell)=1\\ a\pmod { \ell^2}}} e\left( \frac{a}{\ell^2}\right) S(1,0,\ell^2)^2\ell^{\alpha-4}\\
&= O \left( \sqrt{\ell^{\alpha+t}}\right).
\end{align*}

Combining all cases, we see that 
\begin{align}\label{primePowerModuli}
E(e_1,e_2,h_1,h_2,\ell^\alpha) = O \left( \sqrt{(h_1,h_2,\ell^\alpha)\ell^{\alpha}}\right), \text{ if } \alpha \geq 2.
\end{align}

Combining \eqref{primeModuli} and \eqref{primePowerModuli}, we have 
\begin{align}
E(e_1,e_2,h_1,h_2,\ell^\alpha) = O \left( \sqrt{(h_1,h_2,\ell^\alpha) \ell^\alpha}\right), \text{ for all } \alpha \geq 1.\label{AllPrimePower}
\end{align}
	For $E(e_1,e_2,h_1,h_2,d)$, by multiplicativity and \eqref{AllPrimePower}, we have 
	\begin{align*}
	E(e_1,e_2,h_1,h_2,d)= \prod_{ \ell^{\alpha_\ell}\mid\mid d} E(e_1,e_2,h_1,h_2,\ell^{\alpha_\ell})\ll C^{\omega(d)}\sqrt{(h_1,h_2,d)d},
	\end{align*}
	where $C$ is an absolute constant. 
	\end{proof}
\section{Acknowledgement}
The author wishes to thank Valentin Blomer for suggesting the problem and for comments on an earlier draft. The author expresses gratitude to Kyle Pratt for helpful discussions and comments.

	\bibliographystyle{plain}
	\bibliography{DIVISOR.bib}
\end{document}